\documentclass[12pt]{amsart}
\usepackage{amsmath,amsfonts,amssymb}
\usepackage[dvips]{graphicx}
\usepackage[all]{xy}
\usepackage[numbers,square,sort&compress]{natbib} 
\usepackage{xspace}
\usepackage{a4wide}
\graphicspath{{./}{./figs/}{../../figs/}}
\newtheorem{theorem}{Theorem}[section]   
\newtheorem{corollary}[theorem]{Corollary}
        
\newtheorem{proposition}[theorem]{Proposition}
\theoremstyle{definition}
\newtheorem{definition}[theorem]{Definition}  
\theoremstyle{remark}
        
\newtheorem{example}[theorem]{Example}      
\numberwithin{equation}{section}

\def\F{\mathbb{F}}
\def\Z{\mathbb{Z}}
\def\fy{\mathfrak{f}_Y}
\def\Fy{\mathfrak{F}_Y}
\def\y{\mathbf{y}}
\def\SYmod{S_Y\!/\!\!\sim_Y}
\def\simkappa{\!\!\sim_{\kappa}}
\def\sds{{\sf SDS}\xspace}
\def\gds{{\sf GDS}\xspace}
\def\sdss{{\sf SDS}s\xspace}
\def\gdss{{\sf GDS}s\xspace}
\def\Aut{\mathsf{Aut}}
\def\Acyc{\mathsf{Acyc}}
\def\Shift{\ensuremath{\text{\boldmath{$\sigma$}}}}
\def\rev{\ensuremath{\text{\boldmath{$\rho$}}}}
\def\<{\langle}
\def\>{\rangle}
\def\card#1{|#1|}
\def\bigcard#1{\bigl|#1\bigr|}
\def\Circle{\mathsf{Circ}}

\def\Star{\mathsf{Star}}
\def\vset{\mathrm{v}}
\def\eset{\mathrm{e}}
\def\Per{\mathsf{Per}}
\def\nor{\mathsf{nor}}
\def\NOR{\mathsf{Nor}}
\def\parity{\ensuremath{\mathsf{par}}}
\def\Par{\mathsf{Par}}
\def\eg{e.g.\xspace}
\def\ie{i.e.\xspace}
\begin{document}
\title{Cycle Equivalence of Graph Dynamical Systems}
\author{
Matthew~Macauley\qquad\qquad
Henning~S.~Mortveit
}
\date{}

\begin{abstract}
  Graph dynamical systems (\gdss) can be used to describe a wide range
  of distributed, nonlinear phenomena. In this paper we characterize
  \emph{cycle equivalence} of a class of finite \gdss called
  sequential dynamical systems (\sdss). In general, two finite \gdss
  are cycle equivalent if their periodic orbits are isomorphic as
  directed graphs. Sequential dynamical systems may be thought of as
  generalized cellular automata, and use an \emph{update order} to
  construct the dynamical system map.
  The main result of this paper is a characterization of cycle
  equivalence in terms of shifts and reflections of the \sds update
  order. We construct two graphs $C(Y)$ and $D(Y)$ whose components
  describe update orders that give rise to cycle equivalent \sdss. The
  number of components in $C(Y)$ and $D(Y)$ is an upper bound for the
  number of cycle equivalence classes one can obtain, and we enumerate
  these quantities through a recursion relation for several graph
  classes.  The components of these graphs encode dynamical
  neutrality, the component sizes represent periodic orbit structural
  stability, and the number of components can be viewed as a system
  complexity measure.
\end{abstract}

\thanks{This work was partially supported by Fields Institute in
Toronto, Canada.} 
\maketitle


\section{Introduction}
\label{sec:intro}
Sequential dynamical systems (\sdss) were introduced
in~\cite{Barrett:99a,Mortveit:01a}. These are dynamical systems
constructed from $(i)$ a finite undirected graph $Y$ where each vertex
has a state, $(ii)$ a sequence of vertex functions, and $(iii)$ a word
$w$ over the vertex set of $Y$. The \emph{\sds map} is constructed as
the composition of the functions in the order specified by
$w$. As such, they represent a useful framework for describing
distributed phenomena with causal interactions. This paper is about
\emph{cycle equivalence} of finite graph dynamical systems, which we
study in the context of \sdss. Two \sdss are cycle equivalent if their
periodic orbits are isomorphic as directed graphs. We will study how
the update order affects the structure of the periodic orbits, and
thus the long-term behavior of the system.  As an example, we show the
surprising result that if the \gds base graph is a tree then there is
only one possible periodic orbit configuration, and this holds for any
fixed choice of vertex functions. \sds and \sds-like algorithms occur
in many application areas such as~\cite{Karaoz:04,Pautz:02}, and our
results will provide a behavioral complexity measure for these.

This paper is organized as follows. In Section~\ref{sec:terminology}
we describe \sds related terminology and relevant background results
from~\cite{Mortveit:01a,Reidys:98a} on functional and dynamical
equivalence of \sdss. In Section~\ref{sec:main}, we prove one of the
main results of this paper: any two \sdss where the update orders
differ by a cyclic shift are cycle equivalent, and this holds for any
choice of vertex functions. Additionally, when the vertex states are
taken from $\F_2=\{0,1\}$, which is the standard choice in most
studies of cellular automata, then reflections of the update order
also encode cycle equivalent \sdss. We also show how shifts and
reflections of update orders have a natural interpretation in terms of
source-to-sink operations on acyclic orientations of the \gds graph.
In Section~\ref{sec:foundation} we introduce the graphs $C(Y)$ and
$D(Y)$ which form the basis for our analysis and characterization of
cycle equivalence over general graphs. These graphs are examples of
neutral networks, and we characterize some of their structural
properties. We study the functions $\kappa(Y)$ and $\delta(Y)$, which
count the connected components of $C(Y)$ and $D(Y)$, respectively. We
show how $\delta(Y)$ is given in terms of $\kappa(Y)$ and give several
results for the computation of $\kappa(Y)$ with implications to
dynamics. These functions can be regarded as a measure for system
complexity since they are upper bounds for the number of \sds maps up
to cycle equivalence achievable through variations of the update
order. As a computational example we demonstrate how $\kappa(Y)$
increases from $\Theta(n)$ for radius-$1$ rules (the elementary
cellular automaton rules) to $\Theta(n\cdot 2^n)$ for radius-$2$
rules. We also show how the presence of symmetries in the base graph
may allow for significantly improved bounds in certain cases.
In the summary section we show how cycle equivalence of \sdss is
closely related to Coxeter theory. Some of the results that we prove
in this paper have a natural analog when translated into the language
of Coxeter groups. This opens the door to use the rich mathematical
tools and results of Coxeter theory to study sequential dynamical
systems, something that has never been done before.


\section{Background and Definitions}
\label{sec:terminology}

Let $Y$ be a finite undirected graph with vertex set
$\vset[Y]=\{1,\dots,n\}$, and edge set $\eset[Y]$. Since most graphs
in this paper are finite and undirected we simply refer to this class
of graphs as ``graphs'' and specify if needed. The
\emph{$1$-neighborhood} of vertex $v$ in $Y$ is
$B_1(v;Y)=\bigl\{v'\in\vset[Y]\mid \{v,v'\}\in\eset[Y]\bigr\} \cup
\{v\}$, and the \emph{ordered $1$-neighborhood} $n[v]$ of $v$ is the
sequence of vertices from $B_1(v;Y)$ ordered in increasing order. The
\emph{degree} of vertex $v$ is written $d(v)$. Each vertex $v$ is
assigned a state $y_v\in K$ where $K$ is a finite set. In the
following $y_v$ is called a \emph{vertex state} and the $n$-tuple
$\y=(y_1,\dots,y_n)$ is a \emph{system state}.\footnote{In the context
  of, \eg cellular automata a system state is frequently called a
  \emph{configuration}.} We write
\begin{equation}
  \y[v]=(y_{n[v](1)},\ldots,y_{n[v](d(v)+1)})\;,
\end{equation}
for the restriction of the system state to the vertices in $n[v]$, and
let $\y'[v]$ denote the same tuple but with the vertex state $y_v$
omitted. The finite field with $q=p^k$ elements is denoted~$\F_q$.

Let $\fy:=(f_i)_{i\in\vset[Y]}$ be a sequence of \emph{vertex
functions} $f_i\colon K^{d(i)+1}\longrightarrow K$, and define the
sequence of \emph{$Y$-local functions} $\Fy:=(F_i)_{i\in\vset[Y]}$
with $F_i\colon K^n\longrightarrow K^n$ by
\begin{equation}
 F_i(y_1,\cdots,y_n) =
   (y_1,\ldots,y_{i-1}, f_i(\y[i]),y_{i+1},\ldots,y_n)\;.
\end{equation}
It is clear that $\fy$ completely determines $\Fy$, and
vice-versa. However, there are settings when it is easier to speak of
one rather than the other.

Let $W_Y$ denote the set of words over $\vset[Y]$.\footnote{Also
referred to as the Kleene star or Kleene closure of $\vset[Y]$.} Words
are written as $w=(w_1,w_2,\ldots,w_m )$, $w=w_1w_2\cdots w_m$,
$w=(w(1),w(2),\ldots, w(m))$, etc. The subset of $W_Y$ where each
element of $\vset[Y]$ occurs exactly once is denoted $S_Y$. The
elements of $S_Y$ may thus be thought of as permutations of
$\vset[Y]$. The symmetric group $S_n$ acts on system states by
\begin{equation}
  \label{eq:saction}
  \gamma\cdot(y_1,\ldots,y_n)=(y_{\gamma^{-1}(1)},\ldots,y_{\gamma^{-1}(n)})\;.
\end{equation}
\begin{definition}[Sequential dynamical system]
  A \emph{sequential dynamical system} (\sds) is a triple $(Y,\Fy,w)$
  where $Y$ is a graph, $\Fy=(F_i)_{i\in\vset[Y]}$ is a sequence of
  $Y$-local functions, and $w=(w_1,\ldots,w_m)\in W_Y$.
The associated \sds map $[\Fy,w]\colon K^n\longrightarrow K^n$ is the
function composition
\begin{equation}
  [\Fy,w]=F_{w_m}\circ F_{w_{m-1}}\circ\cdots\circ F_{w_2}\circ F_{w_1}\;.
\end{equation}
\end{definition}
The graph $Y$ of an \sds is called the \emph{base graph}, and when
$w\in S_Y$, the \sds is referred to as a \emph{permutation \sds.}  A
sequence of $Y$-local functions $\Fy$ is \emph{$\Aut(Y)$-invariant} if
$\gamma\circ F_v=F_{\gamma(v)}\circ\gamma$ for all $v\in\vset[Y]$ and
all $\gamma\in\Aut(Y)$. Here, the composition of a function $K^n\to
K^n$ with a permutation of $K$ is interpreted as
in~\eqref{eq:saction}. The corresponding sequence of vertex functions
$\fy$ is $\Aut(Y)$-invariant if $\Fy$ is $\Aut(Y)$-invariant.
The \emph{phase space} of the map $\phi\colon K^n \longrightarrow K^n$
is the directed graph $\Gamma(\phi)$ with vertex set $K^n$ and edge
set $\bigl\{(\y,\phi(\y))\mid\y\in K^n\bigr\}$. The following example
illustrates these concepts. 

\begin{example}[Asynchronous Elementary Cellular Automaton rule \#~$1$.] 
  \label{ex:running1}
  Let $Y = \Circle_4$ which is the graph with vertex set $\{1,2,3,4\}$
  and edges all $\{i,i+1\}$ with indices modulo $4$, and take binary
  vertex states. Then $y = (y_1, y_2, y_3, y_4)$, $n[1] = (1,2,4)$,
  and $y[1] = (y_1, y_2, y_4)$. If we use the Boolean $\nor$-function
  $\nor_3 \colon \F_2^3 \longrightarrow\F_2$ (\ie ECA \#~$1$) given by
  $\nor_3(x,y,z) = (1+x)(1+y)(1+z)$ to induce the vertex functions we
  get, \eg $F_1(y) = (\nor_3(y[1]), y_2, y_3, y_4)$. With update order
  $\pi = (1,2,3,4)$ we get the \sds map
  \begin{equation}
    \label{eq:sdsexmp}
    [\NOR_Y, \pi] = F_4 \circ  F_3 \circ F_2 \circ F_1 \;.
  \end{equation}
  It is easy to verify that $[\NOR_Y,\pi](0,0,0,0) = (1, 0, 1, 0)$. In
  contrast, for a parallel update scheme the state $(0,0,0,0)$ would
  have been mapped to $(1,1,1,1)$. The entire phase space of the \sds
  map in~\eqref{eq:sdsexmp} is displayed on the left in
  Figure~\ref{fig:ex:norce}.
\end{example}

What follows is a short overview of functional and dynamical
equivalence of \sdss. The analysis is largely concerned with the
aspect of update order and characterizes \sds maps for a fixed graph
$Y$ and fixed $Y$-local functions $\Fy$ in terms of $w$. It will
provide the basis for cycle equivalence.


\subsection{Functional Equivalence}
\label{sec:fequiv}

Two \sdss are \emph{functionally equivalent} if their \sds maps are
identical as functions. For a fixed sequence $\Fy$, a natural question
to ask is when is $[\Fy,w]=[\Fy,w']$ for $w,w'\in W_Y$. The
\emph{update graph} $\hat{U}(Y)$ provides an answer to this. The
update graph of $Y$ has vertex set $W_Y$. Two words $w\neq w'$ are
adjacent if they have equal length, say $m$, and $(i)$ they differ
\emph{only} by a transposition of entries $k$ and $k+1$, and
$(ii)$ $\{w_k,w_{k+1}\}\not\in \eset[Y]$. The finite subgraph $U(Y)$
of $\hat{U}(Y)$ induced by the vertex set $S_Y$ is called the
\emph{permutation update graph}, and is denoted $U(Y)$. Clearly, it is
a union of connected components of $\hat{U}(Y)$. Both $U(Y)$ and
$\hat{U}(Y)$ are examples of neutral networks as mentioned in the
introduction. The update graph $U(\Circle_4)$ is shown
Figure~\ref{fig:uc4}.

Let $\sim_Y$ be the equivalence relation on $S_Y$ defined by $\pi
\sim_Y \pi'$ iff $\pi$ and $\pi'$ belong to the same connected
component in $U(Y)$. We denote equivalence classes as $[\pi]_Y$ and
the set of equivalence classes by $\SYmod$, \ie
\begin{equation*}
\SYmod = \{ [\pi]_Y \mid \pi\in S_Y \} \;.
\end{equation*}
By construction, $\pi\sim_Y\sigma$ implies the equality $[\Fy,\pi] =
[\Fy,\sigma]$. If the vertex functions are the Boolean functions as in
Example~\ref{ex:running1}, then $[\NOR_Y,\pi] = [\NOR_Y,\sigma]$
implies $\pi \sim_Y \sigma$ (see~\cite{Mortveit:01a}). It follows that
$\card{\SYmod\!\!}$ is a sharp upper bound for the number of
functionally non-equivalent permutation \sds maps obtainable by
varying the update order.

Functional equivalence can also be characterized through acyclic
orientations. An orientation of $Y$ is a map $O_Y \colon
\eset[Y]\longrightarrow \vset[Y] \times \vset[Y]$ that sends an edge
$\{i,j\}$ to either $(i,j)$ or $(j,i)$. 
Let $\Acyc(Y)$ denote the set of acyclic orientations of $Y$, that is,
orientations that contain no directed cycles. In~\cite{Reidys:98a} it
is shown that there is a bijection
\begin{equation}
  \label{eq:abij}
  f_Y \colon \SYmod\,\longrightarrow \Acyc(Y)\;.
\end{equation}
A permutation $\pi\in S_Y$ defines a linear order on $\vset[Y]$ by
$\pi_k = i <_{\pi} j = \pi_\ell$ iff $k<\ell$. This order defines an
acyclic orientation $O_Y^\pi$ where $O_Y^{\pi}(\{v,v'\})$ equals
$(v,v')$ if $v <_\pi v'$ and $(v',v)$ otherwise. The map $f_Y$
in~\eqref{eq:abij} sends $[\pi]_Y\in \SYmod$ to $O_Y^\pi$. It follows
that
\begin{equation}
  \label{eq:abound}
  \alpha(Y)=\card{\Acyc(Y)}
\end{equation}
is a sharp upper bound for the number of functionally non-equivalent
permutation \sdss that can be obtained by varying the update order.
The result can be extended to general word update orders $w\in
W_Y$. We do not review this here, but refer to~\cite{Reidys:04a}.


\subsection{Dynamical Equivalence}

Two finite dynamical systems $\phi,\psi \colon K^n \longrightarrow
K^n$ are \emph{dynamically equivalent} if there exists a bijection $h
\colon K^n \longrightarrow K^n$ such that
\begin{equation}
  \phi \circ h = h \circ \psi \;.
\end{equation}
With the discrete topology the concepts of dynamical equivalence and
topological conjugation coincide. Thus, the difference between
functional and dynamical equivalence is that in the former, the phase
spaces are identical, but in the latter, the phase spaces need just be
isomorphic. 
Update orders that are related by an automorphism of the base graph
give rise to dynamically equivalent \sdss. The number of orbits
$\bar{\alpha}(Y)$ under the action of $\Aut(Y)$ on $\SYmod$ given by
$\gamma \cdot [\pi]_Y = [\gamma *\pi]_Y$, where
\begin{equation}
  \label{eq:action}
  \gamma* w=\gamma(w_1),\dots,\gamma(w_m) \;,
\end{equation}
is an upper bound for the number of \sds maps up to dynamical
equivalence. This follows since for \sdss with $\Aut(Y)$-invariant
vertex functions one has (see~\cite{Mortveit:01a})
\begin{equation}
  \label{eq:gamma*pi}
  [\Fy,\gamma*\pi] \circ \gamma = \gamma \circ [\Fy,\pi]
\end{equation}
for all $\pi\in S_Y$ and all $\gamma\in\Aut(Y)$. Via the bijection
in~\eqref{eq:abij}, this action carries over to an action on the set
$\Acyc(Y)$, and the number of orbits is given by
\begin{equation*}
  \bar{\alpha}(Y) = \frac{1}{\card{\Aut(Y)}} \sum_{\gamma \in \Aut(Y)} 
  \alpha(\<\gamma\> \setminus Y ) \;.
\end{equation*}
Here $\<\gamma\>\setminus Y$ denotes the \emph{orbit graph} of the
cyclic group $G=\<\gamma\>$ and $Y$,
see~\cite{Barrett:01a,Barrett:03a}.
This bound is known to be sharp for certain graph
classes~\cite{Barrett:03a}, but in the general case this is still an
open problem.


\section{Cycle Equivalence}
\label{sec:main}

\begin{definition}
  Two finite dynamical systems $\phi \colon K^n\longrightarrow K^n$
  and $\psi\colon K^m \longrightarrow K^m$ are \emph{cycle
  equivalent}\footnote{In general one can define this where $\phi$ and
    $\psi$ have different vertex state sets $K$ and $K'$.} if there
  exists a bijection $h \colon \Per(\phi) \longrightarrow \Per(\psi)$
  such that
  \begin{equation}
    \label{eq:conj}
    \psi|_{\Per(\psi)} \circ h = h \circ \phi|_{\Per(\phi)} \;,
  \end{equation}
  where $\psi|_{\Per(\psi)}$ and $\phi|_{\Per(\phi)}$ denote the
  restrictions of the maps to their respective sets of periodic points
  $\Per(\psi)$ and $\Per(\phi)$. Two systems $\phi$ and $\psi$ with
  identical periodic orbits are \emph{functionally cycle equivalent}. 
\end{definition}
\begin{example}
\label{ex:running2}
  As an illustration we continue Example~\ref{ex:running1} with
  $Y=\Circle_4$ with vertex functions $\nor_3 \colon\F_2^3
  \longrightarrow\F_2$ with update orders $\pi = (1,2,3,4)$,
  $\pi'=(1,4,2,3)$ and $\pi'' = (1,3,2,4)$.
  \begin{figure}[ht]
    \centerline{
      \includegraphics[width=0.85\textwidth]{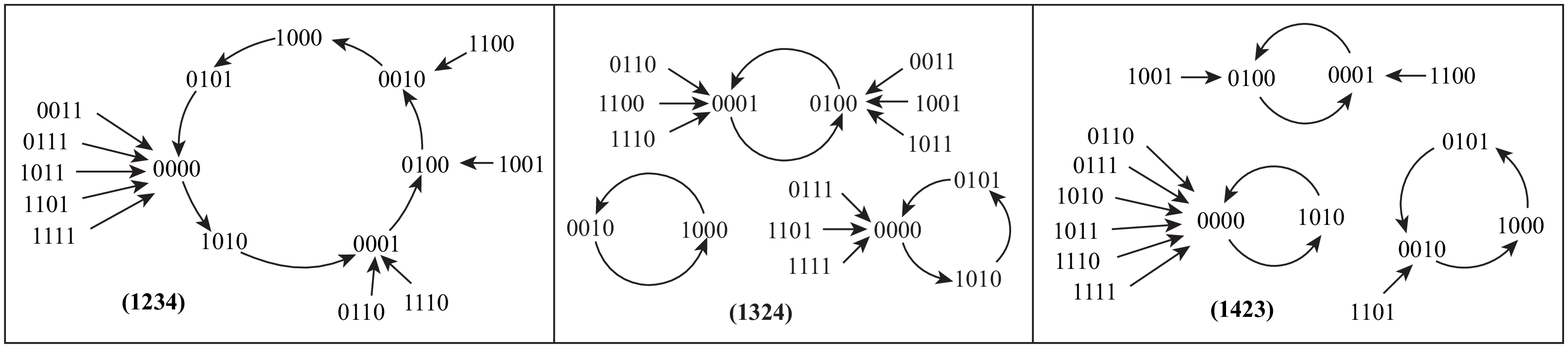}
    }
    \caption{The phase spaces of Example~\ref{ex:running2}.}
    \label{fig:ex:norce}
  \end{figure}
  The two \sds maps $[\NOR_Y,\pi']$ and $[\NOR_Y,\pi'']$ are
  cycle equivalent, which can be seen in the two rightmost phase
  spaces in Figure~\ref{fig:ex:norce}. They are not functionally cycle
  equivalent. Later we show that for $Y=\Circle_4$ there are at
  most $2$ cycle configurations when $K=\F_2 = \{0,1\}$.
\end{example}

It is clear that both functional equivalence and dynamical equivalence
imply cycle equivalence. Define $\sigma,\rho\in S_m$ to be the
permutations
\begin{equation*}
  \sigma = (m, m-1,\ldots,2,1)\;, \qquad 
  \rho = (1,m)(2,m-1)\cdots (\lceil \tfrac{m}{2}\rceil, 
  \lfloor \tfrac{m}{2} \rfloor +1) \;,
\end{equation*}
and let $C_m$ and $D_m$ be the groups
\begin{equation}
  \label{eq:generators}
  C_m=\<\sigma\> \text{\quad and\quad }
  D_m=\<\sigma,\rho\>\;.
\end{equation}
Both $C_m$ and $D_m$ act on the set of length-$m$ update orders
via~\eqref{eq:saction}. Define the $s$-shift $\Shift_s(w)=\sigma^s
\cdot w$, and the reflection $\rev(w)=\rho\cdot w =
(w_m,w_{m-1},\ldots,w_2,w_1)$. We can now state one of the main
results.
\begin{theorem}
  \label{thm:shift}
  For any $w\in W_Y$, the \sds maps $[\Fy, w]$ and $[\Fy,\Shift_s(w)]$
  are cycle equivalent. 
\end{theorem}
\begin{proof}
  Set $P_k = \Per[\Fy, \Shift_k(w)]$.
  By the definition of an \sds map, the following diagram commutes
  \begin{equation}
    \label{eq:conjugates}
    \xymatrix{
      P_{k-1} \ar@{->}[rr]^{[\Fy,\,\Shift_{k-1}(w)]} \ar@{->}[d]_{F_{w(k)}} 
      && P_{k-1} \ar@{->}[d]^{F_{w(k)}} \\ 
      P_k \ar@{->}[rr]_{[\Fy,\,\Shift_k(w)]} &&  P_k
    }
  \end{equation}
  for all $1\le k\le m=\card{w}$. Thus we obtain the inclusion
  $F_{w(k)}(P_{k-1}) \subset P_{k}$, and since the restriction map
  $F_{w(k)} \colon P_{k-1} \longrightarrow F_{w(k)}(P_{k-1})$ is an
  injection, it follows that $\card{P_{k-1}} \le \card{P_k}$. We
  therefore obtain the sequence of inequalities
  \begin{equation*}
    \bigcard{\Per[\Fy,w]}
    \leq \bigcard{\Per[\Fy,\Shift_1(w)]} 
    \leq \cdots 
    \leq \bigcard{\Per[\Fy,\Shift_{m-1}(w)]} 
    \leq \bigcard{\Per[\Fy,w]} \;,
  \end{equation*}
  from which it follows that all inequalities are, in fact,
  equalities. Since the graph and state space are finite all the
  restriction maps $F_{w(k)}$ in~\eqref{eq:conjugates} are
  bijections. Clearly~\eqref{eq:conj} holds with $h = F_{w(k)}$,
  and the proof follows.
\end{proof}

Theorem~\ref{thm:shift} shows that acting on the update order by the
cyclic group $C_m$ preserves the cycle structure of the phase
space. We point out that this result holds for any finite set $K$. For
$K=\F_2$ the cycle structure is also preserved under the action of
$D_m$, and is a consequence of:
\begin{proposition}[\cite{Mortveit:01a}]
  \label{prop:inverse}
  Let $(Y, \Fy, w)$ be an \sds over $\F_2$ with periodic points $P
  \subset \F_2^n$. Then 
  \begin{equation}
    \label{eq:inverse}
    \bigl([\Fy,w]\bigl|_P\bigr)^{-1} = [\Fy, \rev(w)]\bigl|_P\;.
  \end{equation}
\end{proposition}
This result follows since for each vertex function $f_i$ the
restriction $f_i(-;\y'[i]) \colon \F_2 \longrightarrow \F_2$ is a
bijection for each fixed choice of $\y'[i]$. There are only two such
maps: the identity map $y_i\mapsto y_i$ and the map $y_i\mapsto
1+y_i$. From this it follows that composing the two maps
in~\eqref{eq:inverse} in either order gives the identity map,
see~\cite{Mortveit:01a}. The next proposition is now clear:
\begin{proposition}
  \label{prop:oequiv}
  For $K=\F_2$ the \sds maps $[\Fy,w]$ and $[\Fy,\rev(w)]$ are cycle
  equivalent.
\end{proposition}

Thus, for any $g\in G = C_m$ the \sds maps $[\Fy, w]$ and
$[\Fy, g\cdot w]$ are cycle equivalent, where $|w|=m$. If $K=\F_2$
the same statement holds for $G = D_m$.
We now have the following situation: elements $\pi$ and $\pi'$ with
$[\pi]_Y \ne [\pi']_Y$ generally give rise to functionally
non-equivalent \sds maps. If there exists $g\in G$,
$\bar{\pi}\in[\pi]_Y$ and $\bar{\pi}'\in[\pi']_Y$ such that
$g\cdot\bar{\pi} = \bar{\pi}'$, then the classes $[\pi]_Y$ and
$[\pi']_Y$ give rise to cycle equivalent \sds maps.

Let $\Star_n$ be the graph with vertex set $\vset[\Star_n] =
\{0,1,\dots,n\}$ and edge set $\eset[\Star_n]=\bigl\{\{0,i\}\mid 1\leq
i\leq n\bigr\}$. As a particular example we have:
\begin{corollary}
  \label{cor:star}
  Let $Y = \Star_n$ with $n\ge2$. For a fixed sequence $\Fy$ of
  $\Aut(Y)$-invariant $Y$-local maps all permutation \sds maps of the
  form $[\Fy, \pi]$ are cycle equivalent. Any \sds map of the form
  $[\NOR_Y,\pi]$ with $\pi \in S_Y$ has precisely one periodic orbit
  of size three, and $2^{n-1}-1$ periodic orbits of size two.
\end{corollary}
\begin{proof}
  We have $\Aut(\Star_n) \cong S_n$ since the automorphisms of
  $\Star_n$ are precisely the elements of $S_Y$ that fix the vertex
  $0$. An orbit of $\Aut(\Star_n)$ in $\SYmod$ contains all
  equivalence classes $[\pi]_Y$ for which the position of $0$ in $\pi$
  coincide, $i$ say. Thus for $0\le i \le n$ all permutations that
  have vertex $0$ in the $i^{\rm th}$ position give rise to
  dynamically equivalent \sds maps. Pick $\pi = (0,1,2,\ldots,n)$. By
  Corollary~\ref{prop:oequiv}, all permutations that are shifts of
  $\pi$ give cycle equivalent \sds maps. The second part now follows
  by inspection of one of the possible phase spaces. They are all
  listed in~\cite{Mortveit:01a}, but without enumerations of periodic
  orbits.
\end{proof}

\section{Combinatorial Constructions for Cycle Equivalence}
\label{sec:foundation}


\subsection{Neutral Networks}

In the remainder of this paper we will only consider permutation
update orders, although it is not hard to see how this can be extended
to systems with general word update orders. To start, we define two
graphs over $\SYmod$ whose connected components give rise to cycle
equivalent \sdss for a fixed graph $Y$ and a fixed sequence $\Fy$.
Since cycle equivalence is a coarsening of functional equivalence, it
is natural to construct these graphs using $\SYmod$ as vertex set
rather than $S_Y$.

Let $C(Y)$ and $D(Y)$ be the graphs defined by
\begin{alignat*}{3}
  \vset[C(Y)] &= \SYmod ,&\qquad 
  \eset[C(Y)] &= \bigl\{ \{[\pi]_Y,[\Shift_1(\pi)]_Y\} &\;\mid\;&\pi\in
  S_Y\bigr\}\;,\\ 
  \vset[D(Y)] &= \SYmod ,&\qquad 
  \eset[D(Y)] &=  \bigl\{\{[\pi]_Y,\,[\rev(\pi)]_Y\}&\;\mid \;&\pi\in
  S_Y\bigr\} 
  \cup \eset[C(Y)] \;.
\end{alignat*}

Define $\kappa(Y)$ and $\delta(Y)$ to be the number of connected
components of $C(Y)$ and $D(Y)$, respectively. It is clear that $C(Y)$
is a subgraph of $D(Y)$, and that $\delta(Y)\leq \kappa(Y)$. By
Theorem~\ref{thm:shift}, $\kappa(Y)$ is a general upper bound for the
number of different \sds cycle equivalence classes obtainable through
update order variations. For $K=\F_2$ it follows from
Proposition~\ref{prop:inverse} that $\delta(Y)$ is an upper bound as
well. 

\begin{example}
  As in Example~\ref{ex:running1} let $Y=\Circle_4$. The permutation
  update graph $U(\Circle_4)$ has $\alpha(\Circle_4) = 14$ connected
  components as shown in Figure~\ref{fig:uc4}.
  \begin{figure}[ht]
    \centerline{\includegraphics[width=0.6\textwidth]{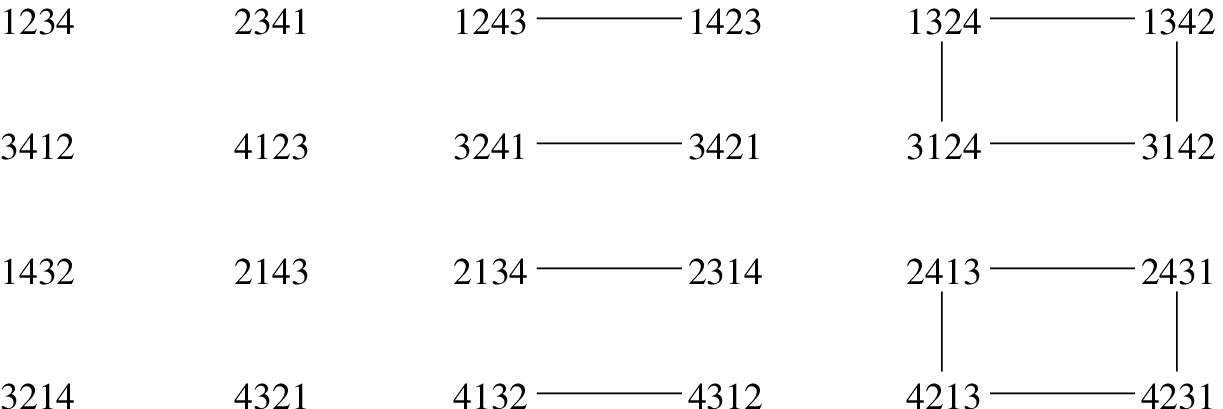}}
    \caption{The update graph $U(\Circle_4)$.}
    \label{fig:uc4} 
  \end{figure}
  The graphs $C(\Circle_4)$ and $D(\Circle_4)$ are shown in
  Figure~\ref{fig:circ4} where the dashed lines represent the edges in
  $\eset[D(\Circle_4)] \setminus \eset[C(\Circle_4)]$.
  \begin{figure}[!htbp]
    \begin{center}
      \includegraphics[width=.7\textwidth]{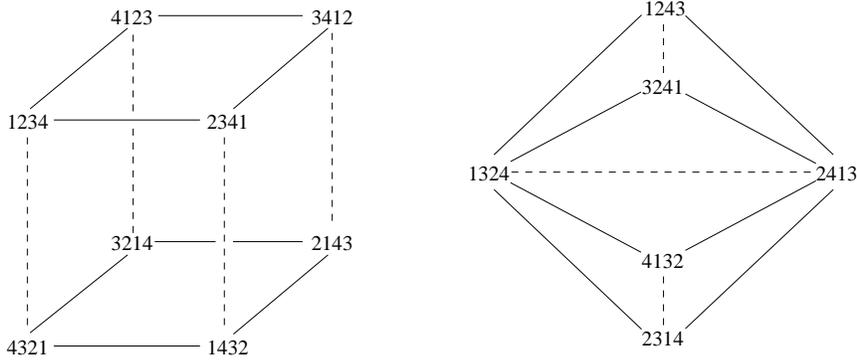}
      \caption{The graphs $C(\Circle_4)$ and $D(\Circle_4)$. The
	dashed lines are edges in $D(\Circle_4)$ but not in
	$C(\Circle_4)$.}
      \label{fig:circ4}
    \end{center}
  \end{figure}
  The vertices in Figure~\ref{fig:circ4} are labeled by a permutation
  in the corresponding equivalence class in $\SYmod$. The vertices of
  the cube-shaped component are all singletons in $\SYmod$. The
  equivalence classes $[1324]_{\Circle_4}$ and $[2413]_{\Circle_4}$
  both consist of four permutations, while the remaining four vertices
  on that component are equivalence classes that contain precisely two
  permutations. Clearly, $\kappa(\Circle_4)=3$ and
  $\delta(\Circle_4)=2$.
\end{example}


The following result gives insight into the how $\kappa$- and
$\delta$-equivalent permutations are distributed across the vertices
of the update graph $U(Y)$. 
\begin{proposition}
  \label{prop:C_n-D_n}
  Let $Y$ be a connected graph on $n$ vertices and let $g,g'\in C_n$
  with $g \ne g'$. Then $[g\cdot\pi]_Y \ne [g'\cdot\pi]_Y$. If
  $g,g'\in D_n$ with $g\ne g'$ then $[g\cdot\pi]_Y=[g'\cdot\pi]_Y$
  holds if and only if $Y$ is bipartite.
\end{proposition}
The proof, which can be found in~\cite{Macauley:08c}, is by
contradiction. We remark that if $Y$ is bipartite if and only if
$\card{\{[g\cdot\pi]_Y : g \in D_n\}}=2n-1$. 

\subsection{Source-Sink Operations and Reflections of Acyclic Orientations}

In this section we show how the component structure of $C(Y)$ is
precisely captured through \emph{source-sink} operations on acyclic
orientations. The bijection in~\eqref{eq:abij} identifies $[\pi]_Y$
with $O_Y^\pi\in\Acyc(Y)$. For any $\pi\in[\pi']_Y$, the orientation
$O^{\Shift_1(\pi)}_Y$ is constructed from $O^\pi_Y$ by converting
vertex $\pi_1$ from a source to a sink. Following~\cite{Shi:01} we
call such a conversion a \emph{source-sink operation} or a
\emph{click}. It can be easily verified that this gives rise to an
equivalence relation $\sim_\kappa$ on $\Acyc(Y)$. More precisely, two
orientations $O_Y,O_Y'\in\Acyc(Y)$ where $O_Y$ can be transformed into
$O'_Y$ by a sequence of clicks are said to be $\kappa$-equivalent.
This observation along with Theorem~\ref{thm:shift} shows that
permutations that belong to $\kappa$-equivalent acyclic orientations
induce cycle equivalent \sdss. By construction, the source-sink
operation precisely encodes adjacency in the graph $C(Y)$, and the
connected components are in 1--1 correspondence with the
$\kappa$-equivalence classes. Therefore, the number of equivalence
classes in $\Acyc(Y)$ under the source-sink relation equals
$\kappa(Y)$, and is thus an upper bound for the number of cycle
equivalent permutation \sds maps $[\Fy,\pi]$ for a fixed sequence
$\Fy$.

If $K=\F_2$ then Proposition~\ref{prop:inverse} shows that reflections
of update orders also induce cycle equivalent \sdss. On the level of
acyclic orientations this corresponds to reversing all orientations.
Through the bijection~\eqref{eq:abij} this identifies $O_Y^{\pi}$ with
the reverse orientation $O_Y^{\rev(\pi)}$, the unique orientation that
satisfies $O_Y^{\pi}(\{i,j\})\neq O_Y^{\rev(\pi)}(\{i,j\})$ for every
$\{i,j\}\in\eset[Y]$. If two acyclic orientations are related by a
sequence of source-sink operations and reflections, then they are said
to be \emph{$\delta$-equivalent}. 

The notion of $\kappa$- and $\delta$-equivalence carries over
naturally to update orders as well. Two update orders in $S_Y$
belonging to $\sim_Y$ classes on the same connected component in
$C(Y)$ [resp.~$D(Y)$] are called \emph{$\kappa$-equivalent}
[resp.~\emph{$\delta$-equivalent}]. For two $\kappa$-equivalent update
orders $\pi$ and $\pi'$, there is a sequence of adjacent non-edge
transpositions and cyclic shifts that map $\pi$ to $\pi'$. This is
simply a consequence of the definition of $\SYmod$ and $C(Y)$.

We remark that from here there is a close connection to the structure
of conjugacy classes of Coxeter elements, something we explain more in
Section~\ref{sec:summary}.  The case of $K=\F_2$ and reflections does
not seem to play any role in Coxeter theory.


\section{Enumeration for $\kappa(Y)$ and $\delta(Y)$}
\label{sec:counting}

It is not difficult to show that $\delta(Y)$ may be characterized in
terms of $\kappa(Y)$.

\begin{proposition}[\cite{Macauley:08c}]
  \label{prop:delta}
  Let $Y$ be a connected graph. If $Y$ is not bipartite then
  $\delta(Y)=\tfrac{1}{2}\kappa(Y)$. If $Y$ is bipartite then
  $\delta(Y)=\tfrac{1}{2}(\kappa(Y)+1)$.
\end{proposition}
The proof uses the fact that $\rev\colon S_Y\longrightarrow S_Y$
extends to an involution
\begin{equation}
  \label{eq:rho-star}
  \rev^* \colon \Acyc(Y)/\simkappa \longrightarrow 
  \Acyc(Y)/\simkappa \;.
\end{equation}
The result now follows since $\rev^*$ has no fixed points if $Y$
is not bipartite, and has precisely one fixed point if $Y$ is
bipartite. As a corollary, a connected graph is bipartite if and only
if $\kappa(Y)$ is odd. In light of Proposition~\ref{prop:delta} we
focus on the computation of $\kappa(Y)$ in the following. It can be
shown that $\kappa(Y)$ does not depend on bridge edges, \ie, edges not
contained in a cycle.

\begin{proposition}[\cite{Macauley:08b}]
\label{prop:dunion}
If $Y$ is the disjoint union of graphs $Y_1$ and $Y_2$, or if $Y$ is a
graph with $e=\{v,w\}$ a bridge connecting the subgraphs $Y_1$ and
$Y_2$, then
\begin{equation}
\label{eq:k-bridge}
\kappa(Y)=\kappa(Y_1)\kappa(Y_2)\;.
\end{equation}
\end{proposition}

For the computation of $\kappa(Y)$ we may therefore assume that $Y$ is
connected, and that every edge is a cycle-edge. Note that for the
empty graph on $n$ vertices $E_n$ we have $\kappa(E_n) = 1$ since
$\alpha(E_n) = 1$.  The following corollary is immediate.

\begin{corollary}
  \label{cor:forest}
  Let $Y$ be a forest. Then $\kappa(Y)=\delta(Y)=1$.
\end{corollary}

>From Corollary~\ref{cor:forest} we get the following perhaps surprising
results on dynamics of \sdss over trees:

\begin{proposition}
  Let $Y$ be a forest and $\Fy$ be a sequence of $Y$-local
  functions. Then all permutation \sds maps $[\Fy,\pi]$ are cycle
  equivalent.
\end{proposition}
The proof is clear since $\kappa$-equivalent update orders induce
cycle equivalent systems. So in other words, when $Y$ is a forest, all
permutation \sdss of the form $[\Fy,\pi]$ for fixed $\Fy$ share the
same cycle configuration. This result may not be that significant if
the \sds only has fixed points, or only has a small number of periodic
points. However, for other functions, such as invertible ones, it is
very powerful. The parity functions
$\parity_k\colon\F_2^k\longrightarrow\F_2$ are defined as
$\parity(\y)=\sum_i y_i$, modulo $2$, and are invertible for every
graph $Y$ (see~\cite{Mortveit:01a}). Let $\Par_Y$ be the sequence of
$Y$-local functions induced by the parity vertex functions.

\begin{corollary}
  If $Y$ is a forest then for any $\pi,\sigma\in S_Y$ the maps
  $[\Par_Y,\pi]$ and $[\Par_Y,\sigma]$ are dynamically equivalent.
\end{corollary}

The same result holds for the logical negation of the parity function,
which is also invertible. 

We now give examples of the computation of $\kappa$. Even though some
of these results may be derived as special cases of more general
results, they are needed for central examples in
Section~\ref{sec:complexity}. We begin with a result for
$\kappa(Y\oplus v)$, the vertex-join of $Y$ and the vertex $v$. Recall
that the graph $Y\oplus v$ is constructed from $Y$ by adding to $Y$
the vertex $v$ and joining this new vertex to each vertex of $Y$.

\begin{proposition}
  \label{prop:vertexjoin}
  If $Y$ is a graph with $\eset[Y]\neq\varnothing$, then
  \begin{equation}
    \label{eq:vjoin}
    \kappa(Y \oplus v)=2\delta(Y \oplus v)=\alpha(Y)\;.
  \end{equation}
\end{proposition}
\begin{proof}
Each $\kappa$-class of $\Acyc(Y\oplus v)$ contains a unique acyclic
orientation where $v$ is a source~\cite{Macauley:08b}. It follows that
there is a bijection between $\Acyc(Y\oplus v)/\simkappa$ and
$\Acyc(Y)$, hence~\eqref{eq:vjoin}. The complete graph on $n$ vertices
is simply the vertex-join of the complete graph on $n-1$ vertices, and
thus we get the following corollary. 
\end{proof}

\begin{corollary}
\label{cor:kn}
  Let $K_n$ denote the complete graph on $n$ vertices. For $n\ge 2$ we
  have $\kappa(K_n)=(n-1)!$.
\end{corollary}

\begin{proof}
  There are $2^{\binom{n}{2}}$ orientations of $K_n$, and by the
  bijection in~\eqref{eq:abij}, precisely $\alpha(K_n)$ of these are
  acyclic, and this is equal to the number of components of the update
  graph $U(K_n)$. Since $U(K_n)$ consists of the $n!$ singleton
  vertices in $S_Y$, $\alpha(K_n)=n!$. By
  Proposition~\ref{prop:vertexjoin},
  $\kappa(K_n)=\alpha(K_{n-1})=(n-1)!$. 
\end{proof}

The quantity $\kappa(Y)$ is in fact a \emph{Tutte-Grothendieck invariant}:

\begin{theorem}[\cite{Macauley:08b}]
\label{thm:recursion}
  Let $e$ be a cycle-edge of $Y$. Then
  \begin{equation}
    \label{eq:recursion}
    \kappa(Y)=\kappa(Y_e')+\kappa(Y_e'')\;,
  \end{equation}
  where $Y_e'$ is the graph obtained from $Y$ by deleting $e$, and
  $Y_e''$ is the graph obtained from $Y$ by contracting $e$. 
\end{theorem}

The proof of Theorem~\ref{thm:recursion} is quite involved, and along
with Proposition~\ref{prop:dunion}, it implies that
$\kappa(Y)=T(Y,1,0)$, where $T(Y,x,y)$ is the Tutte
polynomial~\cite{Tutte:54}. In contrast, it is well-known that the
number of acyclic orientations of a graph satisfies $\alpha(Y) =
T(Y,2,0)$.

A graph $Y$ has an $n$-handle if it is of the form $Y = Y' \cup
\Circle_n$ where $Y'$ and $\Circle_n$ share precisely one edge.
\begin{figure}[ht]
\centerline{\framebox{\includegraphics[scale=0.9]{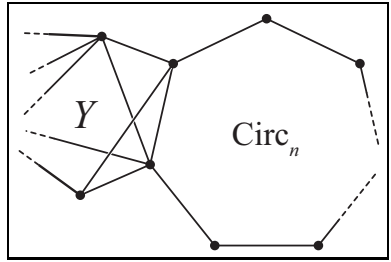}}}
\caption{A graph $Y$ with an $n$-handle.}
\label{fig:handle}
\end{figure}

\begin{proposition}
  \label{prop:k-handle}
  Let $Y$ be a graph with an $n$-handle where $Y = Y' \cup \Circle_n$. Then 
  \begin{equation}
    \label{eq:handle}
    \kappa(Y) = (n-1) \kappa(Y') \;.
  \end{equation}
\end{proposition}
\begin{proof}
  Let $e' = \{v,v'\}$ be the edge shared by $Y'$ and $\Circle_n$ and let
  $e$ be the edge in $\Circle_n$ incident with $v$. By applying
  Theorem~\ref{thm:recursion} and Proposition~\ref{prop:dunion} for
  bridge edges we obtain
  \begin{equation*}
    \kappa( Y' \cup \Circle_n) = \kappa(Y') + \kappa(Y' \cup
    \Circle_{n-1}) \;. 
  \end{equation*}
  Equation~\eqref{eq:handle} follows through repeated applications of
  this process.
\end{proof}

As a simple, special case of Proposition~\ref{prop:k-handle} we obtain
$\kappa(\Circle_n)=n-1$. Just take $Y'$ to be the graph with vertex
set $\vset[Y'] = \{1,n\}$ and edge set $\{\{1,n\}\}$ in
Proposition~\ref{prop:k-handle}.

\section{$\kappa(Y)$ as a Complexity Measure}
\label{sec:complexity}

The number of possible orbit structures that one can obtain by varying
the update order is a natural measure for system complexity. As we
have shown, $\kappa(Y)$ is a general upper bound for this number, and
so is $\delta(Y)$ in the case of binary states. Since these bounds are
graph measures we can characterize complexity in terms of the \gds base
graphs. As we have seen, bridge edges do not contribute to periodic
orbit variability at all, and so it suffices to consider the cycles of
the graph. As can be seen in the case of $\Circle_n$, increasing the
size of a cycle does not contribute much, \eg
$\kappa(\Circle_{n+1})=\kappa(\Circle_n)+1$. However, from the result
on graphs with handles it follows that even the addition of a minimal
handle \emph{doubles} the measure $\kappa$, \ie $\kappa(Y \cup
\Circle_3) = 2 \kappa(Y)$, where $Y$ and $\Circle_3$ share precisely
one edge. The following example shows the effect on complexity that
results from increasing the radius of the rules for elementary
cellular automata.

\begin{example}[CA rule radius vs. periodic orbit complexity]
  \label{ex:radius2}
  We have seen that $\kappa(\Circle_n) = n-1$. Thus, for any fixed
  sequence of radius-$1$ vertex functions the number of distinct
  periodic orbit configurations is $O(n)$. This changes dramatically
  for radius-$2$ rules. In this case the \gds base graph is
  $\Circle_{n,2}$ with
  \begin{equation*}
    \vset[\Circle_{n,2}] = \{1,2,\ldots,n\}, \quad\text{and}\quad
    \eset[\Circle_{n,2}] = \bigl\{  \{i,j \}  \mid 1 \le \card{i-j} \le 2
    \bigr\} \;, 
  \end{equation*}
  with index arithmetic modulo $n$. The auxiliary graph $\Circle_{n,2}'$ is
  obtained from $\Circle_{n,2}$ by deleting the edge $\{2,n\}$. The case
  $n=7$ is illustrated in Figure~\ref{fig:circles}.
  \begin{figure}[ht]
    \centerline{
      \framebox{
	\includegraphics[width=250bp]{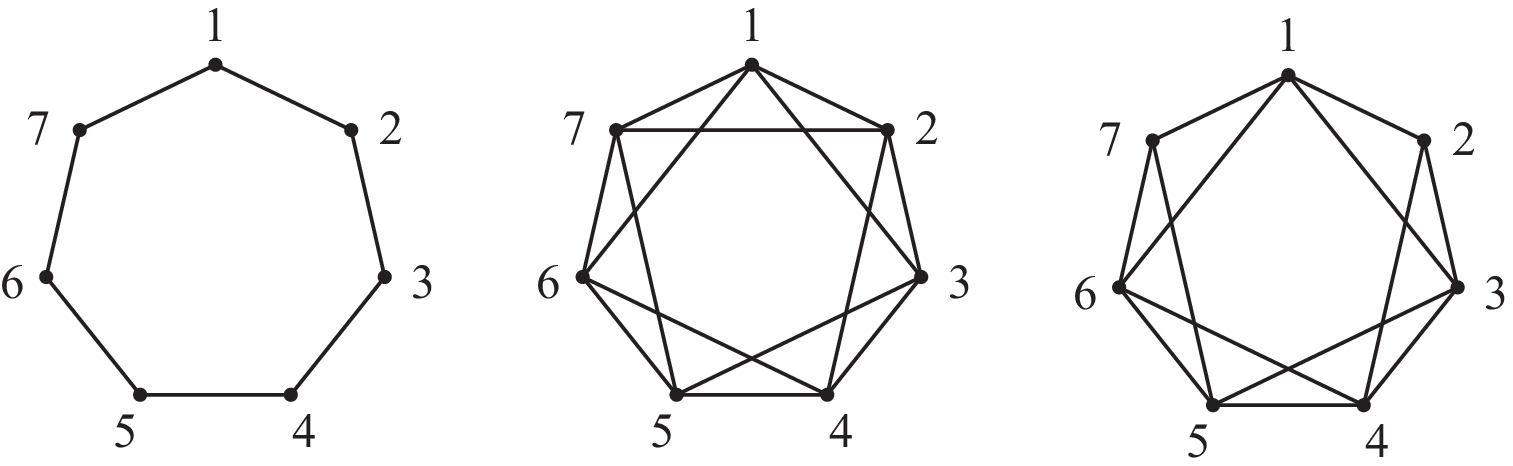}}}
    \caption{From left to right: The graphs $\Circle_7$, $\Circle_{7,2}$
      and $\Circle_{7,2}'$.}
    \label{fig:circles}
  \end{figure}
  
  For simplicity we set $g_n = \kappa(\Circle_{n,2})$ and $c_n =
  \kappa(\Circle_{n,2}')$. Successive uses of the
  recurrence~\eqref{eq:recursion} with edges $e_1 = \{1,n\}$ and $e_2 =
  \{1,n-1\}$ for both $\Circle_{n,2}$ and $\Circle_{n,2}'$ gives
  \begin{equation*}
    c_n = c_{n-1} + 2c_{n-2} + 2^{n-2}\;, \qquad\text{and}\qquad
    g_n = g_{n-2} + c_n + 2 c_{n-2}\;,
  \end{equation*}
  where $c_5 = 18$, $c_6 = 46$, $g_5 = 24$, and $g_6 = 64$. These
  recurrence relations are straightforward to solve with
  \begin{align*}
    \kappa(\Circle_{n,2}') &= \bigl[(3n-5) 2^n -
      4(-1)^n\bigr]/18\text{\quad and}\\
    \kappa(\Circle_{n,2}) &= \bigl[ (2n-6)2^n + 9 - (2n-3)(-1)^n \bigr]/6 \;.
  \end{align*}
  Thus, by increasing the rule radius from $1$ to $2$ we see that the
  number of distinct periodic orbit configurations is $O(n\cdot
  2^n)$. The corresponding bounds for $\delta$ are easily obtained
  from Proposition~\ref{prop:delta}.
\end{example}

We have seen how non-trivial symmetries in the base graph give rise to
dynamically equivalent \sds maps when the vertex functions are
$\Aut(Y)$-invariant. Since dynamical equivalence implies cycle
equivalence we can construct a bound $\bar{\kappa}(Y)$ in the same
manner as for $\bar{\alpha}(Y)$. This bound $\bar{\kappa}(Y)$ thus
reflects the added cycle equivalence that are due to symmetries and
that arise for $\Aut(Y)$-invariant vertex functions. 

We close with an example that illustrates this and the results of
Theorem~\ref{thm:recursion}, and Propositions~\ref{prop:dunion}
and~\ref{prop:vertexjoin} and~\ref{prop:k-handle}.

\begin{example}
  \label{eq:q23}
  Let $Y=Q_2^3$ be the binary $3$-cube, which has automorphism
  group isomorphic to $S_4\times \Z_2$. It is shown
  in~\cite{Barrett:01a} that $\alpha(Q_2^3) = 1862$ and that
  $\bar{\alpha}(Q_2^3)=54$. Thus, there are at most $1862$
  functionally nonequivalent permutation \sdss over $Q_2^3$ for a
  fixed sequence of vertex functions. Likewise, there are at most $54$
  dynamically nonequivalent $\Aut(Q_2^3)$-invariant permutation
  \sdss. It is known that the bound $\bar{\alpha}(Q_2^3)$ is sharp,
  since it is realized for \sdss induced by, \eg the
  $\nor_4$-function.  

  The number of cycle equivalence classes is bounded above by
  $\kappa(Q_2^3)$, and from the recursion
  relation~\eqref{eq:recursion} we get (with some foresight at each
  step)
  \def\gfig#1{\kappa(\raise-7pt\hbox{\includegraphics[width=20pt]{#1}})}
  \begin{align*}
    \gfig{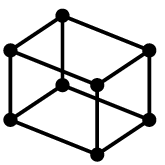} &=  \gfig{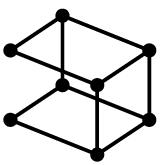} + \gfig{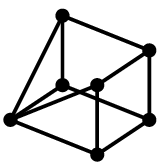} 
    = \gfig{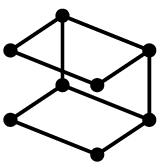} + 2\gfig{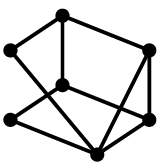} + \gfig{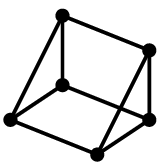} \\
    &= \gfig{c04} + 2\gfig{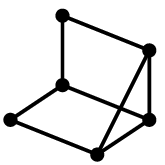} + 2\gfig{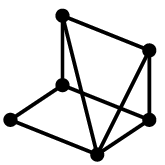} + \gfig{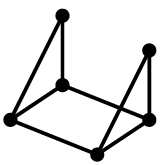} + \gfig{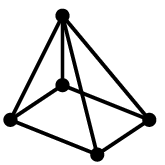}  \\
    &= \gfig{c04} + 4\gfig{c10} + 2\gfig{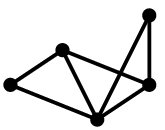}  + \gfig{c08} + \gfig{c09}  \\
    &= 27 + 64 + 16 + 12 + 14 = 133 \;,
  \end{align*}
  where Propositions~~\ref{prop:vertexjoin} and~\ref{prop:k-handle}
  were used in the last step. Since $Q_2^3$ is bipartite we also
  derive $\delta(Q_2^3) = (133+1)/2 = 67$, and thus in the case of
  $K=\F_2$ there are at most $67$ cycle classes for a fixed
  sequence of vertex functions. Straightforward (but somewhat lengthy)
  calculations show that $\bar{\kappa}(Q_2^3)=\bar{\delta}(Q_2^3)=8$.
  In conclusion, we have
  \[
  \alpha(Q_2^3)=1862\,,\quad\bar{\alpha}(Q_2^3)=54\,,\quad
  \kappa(Q_2^3)=133\,,\quad\delta(Q_2^3)=67\,,\quad
  \bar{\kappa}(Q_2^3) = \bar{\delta}(Q_2^3)=8\;.
  \]
  Thus if $\Fy$ is a sequence of $\Aut(Q_2^3)$-invariant $Y$-local
  functions, then there are at most eight different periodic orbit
  configurations for permutation \sds maps $[\Fy,\pi]$ up to
  isomorphism. Moreover, because $\bar{\kappa}(Q_2^3) =
  \bar{\delta}(Q_2^3)$ taking vertex states from $K = \F_2$ does not
  improve this bound.
\end{example}

This example is only meant as an illustration, and a systematic
treatment incorporating the analysis of the functions $\bar{\kappa}$
and $\bar{\delta}$ for general graphs will be pursued elsewhere.


\section{Summary}
\label{sec:summary}

In this paper we have shown how shifts and reflections of update
orders give rise to sequential dynamical systems with isomorphic
periodic orbit configurations. We have also shown how to bound the
number of periodic orbit configurations, and have derived several
properties of this bound $\kappa(Y)$. For binary states we have shown
how $\delta(Y)$ applies to give a sharper bound. Both quantities
$\kappa$ and $\delta$ serve as measures for dynamical complexity.

We also have $\kappa(Y) = T(Y,1,0)$ where $T$ denotes the Tutte
polynomial. There are other mathematical quantities counted by
$T(Y,1,0)$, and thus by $\kappa(Y)$. For example, the set $\Acyc_v(Y)$
consisting of all the acyclic orientation of $Y$ with $v$ as a unique
source is also counted by $\kappa(Y)$, see~\cite{Gioan:07}. In fact,
for each $\kappa$-equivalence class there is unique acyclic
orientation with $v$ as the only source. This allows one to construct
a complete set of representatives for permutations realizing the
possible periodic orbit configurations, see~\cite{Macauley:08b}.

The notion of source-to-sink conversions also shows up in the context
of Coxeter theory (see, \eg~\cite{Bjorner:05} for definitions). For a
Coxeter group with Coxeter graph $Y$ the number of conjugacy classes
of Coxeter elements (see~\cite{Shi:97a}) is also bounded above by
$\kappa(Y)$, \eg~\cite{Shi:01}. In general it is not known if
$\kappa(Y)$ is a sharp bound, but is is known for special classes of
graphs such as $\Circle_n$ as shown by Shi in~\cite{Shi:01}. This
connection between Coxeter theory and \sdss could potentially be very
helpful in further exploring the properties of asynchronous \gdss.

In this paper we have not explored the question of when $\kappa$
(and $\delta$ when $K=\F_2$) is a sharp bound.  That is, for an
arbitrary graph $Y$ does there exist a sequence of vertex function
whose number of non-equivalent orbit configurations equals $\kappa(Y)$?
Proving this would require one to construct such functions for any
given graph. We have also omitted computational aspects related to
cycle equivalence. Given two \sdss, what is the computational
complexity of determining if they are cycle equivalent?  Related
questions have been asked for, \eg fixed point reachability
in~\cite{Barrett:01e}, but see also~\cite{Barrett:06a} for similar
questions. Additional future work includes extending our results from
permutations update orders to general word update orders as well as
further exploring the effects of symmetries in the graph and the
computation of the bounds $\bar{\kappa}$ and $\bar{\delta}$ as
illustrated in Example~\ref{eq:q23}.


\medskip\noindent
{\bfseries Acknowledgments}

The first author would like to thank Jon McCammond for many helpful
discussions. Both authors are grateful to the NDSSL group at Virginia
Tech for the support of this research. This work was partially
supported by Fields Institute in Toronto, Canada. 

\end{document}